\documentclass[11pt, a4paper]{amsart} 
\pdfoutput=1

\title[Invariance properties of MMM characteristic numbers]{Invariance properties of Miller--Morita--Mumford characteristic numbers of fibre
  bundles}
\author{Thomas Church}
\author{Martin Crossley}
\author{Jeffrey Giansiracusa} 
\date{15 December, 2011}

\usepackage{mathptmx}
\DeclareSymbolFont{cmlargesymbols}{OMX}{cmex}{m}{n}
\DeclareMathSymbol{\mycoprod}{\mathop}{cmlargesymbols}{"60}
\let\coprod\mycoprod

\usepackage[protrusion=true,expansion=true]{microtype}

\usepackage{amssymb}
\usepackage{mathrsfs}  %some extra fonts
\usepackage{latexsym}
\usepackage{amsmath}

\usepackage[all,cmtip,arrow]{xy}
\usepackage{pb-diagram,pb-xy}
\usepackage[vmargin=3.5cm, hmargin=3cm]{geometry}
\numberwithin{equation}{subsection}

\newtheorem{thmA}{Theorem}

\newtheorem{corA}[thmA]{Corollary}

\newtheorem{theorem}{Theorem}[subsection]  %theorems numbered within subsections
\newtheorem{lemma}[theorem]{Lemma} 
\newtheorem{proposition}[theorem]{Proposition}

\newtheorem{definition}[theorem]{Definition}

\theoremstyle{remark} % styled differently... not italicized

\newtheorem{question}[theorem]{Question}

\newcommand{\Hom}{\mathrm{Hom}}
\newcommand{\Diff}{\mathrm{Diff}}

\newcommand{\C}{\mathbb{C}}
\newcommand{\R}{\mathbb{R}}
\newcommand{\N}{\mathbb{N}}
\newcommand{\Z}{\mathbb{Z}}
\newcommand{\Q}{\mathbb{Q}}
\newcommand{\Smash}{\wedge}
\newcommand{\alg}{\mathscr{C}}
\DeclareMathOperator*{\colim}{colim} %* means limits directly above/below
\newcommand{\MSO}{\mathbf{MSO}}
\newcommand{\MU}{\mathbf{MU}}
\newcommand{\MTSO}{\mathbf{MTSO}}
\newcommand{\MTU}{\mathbf{MTU}}
\newcommand{\ThS}{\mathbf{Th}}
\newcommand{\Th}{\mathrm{Th}}

\DeclareMathOperator{\hocolim}{\mathrm{hocolim}}

\newcommand{\defeq}{\mathbin{\mathpalette{\vcenter{\hbox{$:$}}}=}}

\def\co{\colon\thinspace} 

\begin{document}
\begin{abstract}
  Characteristic classes of fibre bundles $E^{d+n}\to B^n$ in the
  category of closed oriented manifolds give rise to characteristic
  numbers by integrating the classes over the base.  Church, Farb and
  Thibault \cite{CFT} raised the question of which generalised
  Miller--Morita--Mumford classes have the property that the associated
  characteristic number is independent of the fibering and depends
  only on the cobordism class of the total space $E$.  Here we
  determine a complete answer to this question in both the oriented
  category and the stably almost complex category.  An MMM class has
  this property if and only if it is a fibre integral of a vector
  bundle characteristic class that satisfies a certain approximate
  version of the additivity of the Chern character.
\end{abstract}
\maketitle

\section{Introduction}
\subsection{Background}
Let $\pi\co E^{n+d} \to B^n$ be a fibre bundle with closed oriented
smooth $d$-dimensional manifold fibres. The generalised
Miller--Morita--Mumford (MMM) characteristic classes for such bundles
are a family of characteristic classes living in the cohomology of the
base.  They were first introduced in \cite{Miller}, \cite{Morita}, and
\cite{Mumford}, and there is an extensive literature on these classes
when $d=2$, including the proof of Mumford's Conjecture and work on
Faber's Conjectures.  In higher dimensions these classes are somewhat
less studied, though there are some interesting results, such as those
of \cite{Ebert-vanishing}, \cite{Ebert-alg-indep}.

In \cite{CFT} the first author proved with Farb and Thibault that, in the
$d=2$ setting, certain characteristic numbers associated with the MMM
classes could be expressed purely in terms of the oriented bordism
class of the total space $E$.  The purpose of this paper is to explore
this phenomenon more thoroughly from the point of view of stable
homotopy and bordism theory.

The generalised MMM classes are defined as follows.  A
characteristic class $X \in H^*(BSO(d);\Z)$ for rank $d$ vector
bundles (rationally $X$ can be written as a polynomial in the Euler
class and Pontrjagin classes $p_1, \ldots, p_{\lfloor d/2 \rfloor}$)
can be evaluated on the fibrewise tangent bundle $T^vE \to E$ to
produce a cohomology class $X(T^vE) \in H^*(E;\Z)$. The image of this
class under the pushforward map $\pi_!\co H^*(E;\Z) \to H^{*-d}(B;\Z)$
is a class
\[
\widehat{X}(E\stackrel{\pi}{\to} B) \defeq \pi_! X(T^vE) \in
H^*(B;\Z)
\] 
(note that $\widehat{X}$ is zero if $\mathrm{deg}(X) < d$) that is
natural with respect to pullbacks of bundles.  Products and sums of
these classes are also natural with respect to pullbacks.  The
polynomials in the classes $\widehat{X}$ produced by this construction
are, by definition, the \emph{generalised MMM classes} for
$d$-dimensional oriented fibre bundles.  These classes form an algebra
$\alg_d$ (over $\Z$).  More precisely, $\alg_d$ is the free graded
commutative algebra generated by $H^{>d}(BSO(d);\Z)$ (with degrees
shifted down by $d$ so that a degree $d+i$ class gives a generator of
degree $i$).  We promote it to a bialgebra by declaring the generators
to be primitive.

If the base $B$ is a closed oriented manifold of dimension $n$ and $x
\in \alg_d$ is of degree $n$ then we define an associated \emph{characteristic
  number} $x^\sharp$ by pairing with the fundamental class of the base,
\[
x^\sharp(E\to B) \defeq \langle x(E\to B), [B]\rangle \in \Z.
\]
The characteristic class $x(E\to B)$ is an invariant living in a group
that depends on the base $B$, and so the characteristic classes of
bundles over different bases cannot be directly compared. However, the
associated characteristic numbers, being simply integers, do allow for
immediate comparison.

At this point it is useful to observe that the characteristic number
$x^\sharp(E\to B)$ depends only on the class of $x$ in $\alg_d \otimes
\Q$; the 2-torsion classes in $H^*(BSO(d);\Z)$ (there is no other
torsion) yield MMM classes that necessarily give zero when integrated
over the base.  Thus, when studying MMM characteristic numbers, one
loses no information by working rationally, and we shall do this
from here onwards.

It sometimes happens that a given manifold $E^{d+n}$ admits multiple
distinct fiberings $E\to B$ over bases of a fixed dimension.  In
\cite{CFT} the following question is raised.  \textit{To what extent
  do the MMM characteristic numbers actually depend on the fibering?}
The main result of that paper is that, in the case of surface bundles,
if $x$ is the MMM class $e_{2n-1} \defeq \pi_!  (e(T^vE)^{2n}) \in
H^{2n-2}(B)$ determined by an even power of the Euler class then
$x^\sharp$ can be written as a polynomial in the Pontrjagin numbers of
$E$ and thus depends only on the oriented topological cobordism class
of $E$ and is completely independent of the fibering.

In light of the question and results of \cite{CFT}, it then is natural
to try to decide which of the generalised MMM classes have associated
characteristic numbers that can be expressed purely as functions of
the bordism class of the total space.  To state this problem more
precisely we introduce some notation.  Let $Bun_d$ denote the set of
isomorphism classes of bundles with $d$-dimensional fibres in the
category of smooth closed oriented manifolds; this is a commutative
monoid with respect to disjoint union.  Sending a bundle to the
bordism class of its total space defines an additive map $Tot\co Bun_d
\to MSO_{*}(pt)$.  A class $x\in \alg_d\otimes \Q$ determines an
additive map $x^\sharp\co Bun_d \to \Q$. In this paper we give a
complete answer to the following question.
\begin{question}\label{main-Q}
  For which classes $x\in \alg_d\otimes \Q$ does there exist a
  factorisation as below?
\[
\xymatrix{
& MSO_{*}(pt) \ar@{.>}[d] \\
Bun_d \ar[ur]^{Tot} \ar[r]_{x^\sharp} & \Q}
\]
\end{question}

\subsection{The motivating example}\label{prototype}
Let us first observe that there is an obvious family of elements in
$\alg_d$ for which the associated characteristic numbers depend only
on the total space of the bundle.  Recall that $H^*(BSO;\Q)$ is a Hopf
algebra with coproduct $\Delta$ induced by the direct sum map.
Suppose $X \in H^*(BSO;\Q)$ is primitive (meaning that $\Delta(X) = X
\otimes 1 + 1 \otimes X$) and consider the corresponding generalised
MMM class $\widehat{X} \in \alg_d\otimes \Q$ and its associated
characteristic number $\widehat{X}^\sharp$.  Given a bundle $\pi\co
E^{n+d} \to B^n$, there is a short exact sequence $T^vE \to TE \to
\pi^*TB$  and hence $X(TE) = X(T^vE) + \pi^*X(TB)$ since $X$ is
primitive. We then have,
\begin{align*}
\langle X(TE), [E] \rangle 
  & = \langle \pi_! X(TE), [B]\rangle \\
  & = \langle \pi_!X(T^vE), [B] \rangle + \langle \pi_! \pi^*X(TB), [B]\rangle \\
  & = \widehat{X}^\sharp(E\to B)
\end{align*}
since $\pi_! \circ \pi^*$ is the zero map.  Thus
$\widehat{X}^\sharp(E\to B)$ can be expressed as a polynomial in the
tangential Pontrjagin numbers of $E$ which depend only on the class
of $E$ in oriented bordism tensored with $\Q$.

The primitives in $H^*(BSO;\Q)$ are easily described.  The Pontrjagin
character $Ph \in H^*(BSO;\Q)$ is the pullback of the Chern character
$Ch$ via the complexification map $BSO \to BU$.  The Chern character
is additive: $Ch(V\oplus W) = Ch(V) + Ch(W)$, and so the Pontrjagin
character is also additive; i.e., the components are primitive.  The
space of primitives in $H^*(BU;\Q)$ and $H^*(BSO;\Q)$ is spanned by
the components of the Chern character and Pontrjagin character,
respectively, and the primitives in fact freely generate the
cohomology as a ring since the cohomology is a commutative and
cocommutative connected Hopf algebra.

When $d=2$,
\[
\alg_2 \cong \Z[e_1, e_2, \ldots],
\]
where $e_i = \widehat{e^{i+1}}$. The restriction of the Pontrjagin
character to $BSO(2)$ is zero in degrees $4i+2$ and is proportional to
$e^{2i} = p_1^i$ in degree $4i$. Thus the $e_{odd}$ characteristic
numbers are invariants of the total space, recovering the main theorem
of \cite{CFT}.  It turns out, as we shall see in Theorem \ref{dim2-thm}
below, that these are the only characteristic numbers having this
property.

\subsection{The method}
The main idea of this paper is to rephrase Question \ref{main-Q} in
terms of stable homotopy and certain bordism spectra.  We first show
that the maps $x^\sharp\co Bun_d \to \Q$ and $Tot\co Bun_d \to
MSO_*(pt)$ both admit canonical factorisations through an additive map
\[
\alpha\co Bun_d \to MSO_*(\Omega^\infty \MTSO(d)),
\]
where $\MTSO(d)$ is the Madsen--Tillmann spectrum (defined in Section
\ref{section:thom-spectra} below); we describe all of these maps at
the level of spectra.  In the diagram,
\begin{equation}\label{factorisation-diagram}
\xymatrix{
 & & MSO_{*+d}(pt) \ar@{.>}[dd]\\
Bun_d \ar@/_1pc/[drr]_{x^\sharp} \ar@/^1pc/[urr]^{Tot}
\ar[r]^-{\alpha} & MSO_*(\Omega^\infty \MTSO(d)) \ar[ru]_{\widetilde{Tot}}
\ar[dr]^{x^\sharp} \\
& & \Q,}
\end{equation}
a straightforward rational calculation in stable homotopy then
determines exactly when the dotted arrow can be completed to make the
inner triangle commute.  Given this, we show that the results of Ebert
\cite{Ebert-alg-indep} determine enough information about the image of
$\alpha$ to decide exactly when the dotted arrow can be filled to make
the outer triangle commute.

\subsection{Near-primitive characteristic classes of vector bundles}
In carrying out the rational calculation in Section
\ref{section:calculation} we find that the argument given above in
Section \ref{prototype} with the Pontrjagin character (when
reformulated in terms of stable homotopy) actually yields the
factorisation asked for in Question \ref{main-Q} for a class of
elements in $H^*(BSO;\Q)$ that is somewhat larger than the space of
primitives; we call these new elements \emph{near-primitives}.

\begin{definition}
  Let $H$ be a graded connected cocommutative Hopf algebra.  A
  \emph{near-primitive of order $d$} in $H$ is an element $x$ of
  degree at least $d$ such that the composition
\[
H \stackrel{\Delta}{\to} H\otimes H \stackrel{id\otimes \mathrm{proj}
}{\longrightarrow} H\otimes H/H^{< d}
\]
sends $x$ to $1\otimes x$. 
\end{definition} 
Note that if $H$ is concentrated in nonnegative degrees then a
near-primitive of order $1$ is simply a primitive.  In Section
\ref{classify-near-primitives} we give a complete classification of
near primitives in $H^*(BSO;\Q)$ and $H^*(BU;\Q)$.
\begin{thmA}\label{near-prim-thm}
  The space of order $d$ near-primitives in $H^*(BSO;\Q)$ or
  $H^*(BU;\Q)$ has a basis consisting of those monomials in the
  primitives such that all proper factors have degree strictly less
  than $d$.  Explicitly, the degree $m$ component of the space is
  equal to the degree $m$ part of $\Q[Q_i \:\:|\:\: m-d < |Q_i| < d
  ]\oplus \Q Q_m$, where the $Q_i$ are the components of the
  Pontrjagin character or Chern character respectively.
\end{thmA}
A primitive is a near-primitive (of any order); conversely, it follows
from the theorem that the \emph{only} near-primitives of order $d$ in
degrees $\geq 2d$ are the ordinary primitives.  In particular, for
$BSO$, when $d=2$ all near-primitives are primitive and their
restrictions to $BSO(2)$ are proportional to the odd powers of the
Euler class.  However, in the range $d\leq*<2d$ the space of
near-primitives can be significantly larger than the subspace of
primitives.  The simplest example of near-primitives that are not
primitive occurs when $5\leq d \leq 8$; in this case all elements of
degree 8 (all elements in the span of $p_1^2$ and $p_2$) are
near-primitive of order $d$, while only the scalar multiples of $p_2 -
(1/2)p_1^2$ are primitive.

\subsection{Results}
The result of our rational calculation is the following determination
of when the inner triangle in diagram \eqref{factorisation-diagram}
can be completed.
\begin{thmA}\label{submain-thm}
  Given a class $x \in \alg_d$, the map $x^\sharp\co
  MSO_*(\Omega^\infty \MTSO(d)) \to \Q$ factors through $\widetilde{Tot}\co
  MSO_*(\Omega^\infty \MTSO(d)) \to MSO_*(pt)$ if and only if $x$
  satisfies the following conditions:
\begin{enumerate}
\item $x$ is primitive, so it is of the form
  $\widehat{X}$ for some $X \in H^*(BSO(d);\Q)$, and
\item $X$ is the restriction of a near-primitive of order $d$ in $H^*(BSO;\Q)$.
\end{enumerate}
\end{thmA}

We now turn to the outer triangle in diagram
\eqref{factorisation-diagram}.  Let $L \subset H^*(BSO(d);\Q)$ denote
the subspace spanned by the components of the Hirzebruch $L$-class
(restricted to $BSO(d)$) and let $K$ denote the ideal in
$\alg_d\otimes\Q$ generated by the elements of the form
$\widehat{X}$ for $X\in L$.  Let $NP_d \subset H^*(BSO(d);\Q)$
denote the image of the space of near-primitives of order $d$ on $BSO$
under restriction from $BSO$ to $BSO(d)$.  Our main result is the
following:
\begin{thmA}\label{main-thm}
  Given a class $x \in \alg_d\otimes\Q$, the map $x^\sharp\co Bun_d
  \to \Q$ factors through $Tot\co Bun_d \to MSO_*(pt)$ if and only if
  there exists $X\in NP_d$ such that $x = \widehat{X}$ in $\alg_d$
  when $d$ is even and in $\alg_d/K$ when $d$ is odd.
\end{thmA}

As a special case of this theorem, when $d=2$ we find that the
sufficient condition given in \cite{CFT} is in fact also necessary.
\begin{corA}\label{dim2-thm}
  Given $x \in \alg_2\otimes\Q \cong \Q[e_1, e_2, \ldots]$, the
  characteristic number $x^\sharp\co Bun_2 \to \Q$ factors through
  $Tot\co Bun_2 \to MSO_*(pt)$ if and only if $x$ is in the span of
  the odd MMM class $\{e_{2i+1}\}$.
\end{corA}

\subsection{Almost complex bundles}
There is an analogous story in the setting of almost complex fibre
bundles.  Consider the set $Bun^\C_{d}$ of smooth fibre bundles
$\pi\co E \to B$ in which
\begin{enumerate}
\item the fibrewise tangent bundle $T^vE$ is equipped with a complex
  structure,
\item both $B$ and $E$ are closed oriented stably almost complex
  manifolds (i.e., they have almost complex structures on their stable
  tangent bundles),
\item and the short exact sequence $T^vE \to TE \to \pi^*TB$ of real
  bundles becomes a short exact sequence of complex vector bundles
  after an appropriate stabilisation.
\end{enumerate}

Replacing $SO(d)$ with $U(d)$, we let $\alg^\C_d$ denote the algebra
of generalised MMM classes in this context.  For any $x\in \alg_d^\C$
we consider the diagram
\begin{equation}\label{complex-factorisation-diagram}
\xymatrix{
 & & MU_{*+2d}(pt) \ar@{.>}[dd]\\
Bun_d^\C \ar@/_1pc/[drr]_{x^\sharp} \ar@/^1pc/[urr]^{Tot}
\ar[r]^-{\alpha} & MU_*(\Omega^\infty \MTU(d)) \ar[ru]_{\widetilde{Tot}}
\ar[dr]^{x^\sharp} \\
& & \Q,}
\end{equation}
and we find a corresponding theorem.
\begin{thmA}\label{complex-thm} 
For any $x \in \alg^\C_d$ the following statements hold.
\begin{enumerate}
\item The map $x^\sharp\co MU_*(\Omega^\infty \MTU(d)) \to \Q$ factors
  through the map $\widetilde{Tot}\co MU_*(\Omega^\infty \MTU(d)) \to
  MU_*(pt)$ if and only if $x$ is of the form $\widehat{X}$ for some
  $X \in H^*(BU(d);\Q)$ that is the restriction of a near-primitive of
  order $2d$ in $H^*(BU;\Q)$.
\item The same is true for factoring the map $x^\sharp\co Bun_d^\C \to
  \Q$ through $Tot\co Bun_d \to MU_*(pt)$.
\end{enumerate}
\end{thmA}

Consider the special case of the above theorem for holomorphic surface
bundles, i.e., when $d=1$. As in the real oriented case, $\alpha^\C\co
Bun_d^\C \to MU_*(\Omega^\infty \MTU(1))$ is surjective and $\alg_1^\C
= \Z[e_1, e_2, \ldots]$, but now $e^{i}$ is proportional to a
component of the Chern character no matter whether $i$ is even or odd
--- in fact, $Ch|_{BU(1)} = \sum_n (1/n!)e^n$. Thus all of the
primitive classes in $\alg^\C_1$ give characteristic numbers for
holomorphic surface bundles that depend only on the complex cobordism
class of the total space, and these are the only classes with this
property.

\section{Thom spectra, Madsen--Tillmann spectra and
  bordism}\label{section:thom-spectra}

We begin by recalling some well-known properties of Thom spectra and
their relation to bordism theory.

\subsection{Thom spectra}
Recall that the Thom space $\Th(\xi)$ of a vector bundle $\xi\to B$ is
the disc bundle modulo its boundary sphere bundle, $D(\xi)/S(\xi)$.
Since the Thom space of a rank $n$ trivial bundle is homeomorphic to
the $n$-fold suspension of the base, we define the Thom spectrum
$\ThS(\xi-\R^n)$ of a virtual bundle of the form $\xi - \R^n$ to be
$\Sigma^{-n}\Th(\xi)$; this construction extends to arbitrary
$KO$-theory classes over CW bases via the skeletal filtration of the
base.  The Thom space and Thom spectrum send external products to
smash products, and they are functorial with respect to bundle
pullbacks: given a map $f\co X \to Y$ and a class $\zeta \in KO(Y)$,
there is a map $\ThS(f^*\zeta) \to \ThS(\zeta)$.  Since a sum $\zeta
\oplus \eta$ is the pullback of the external product $\zeta\times
\eta$ on $X\times X$ by the diagonal, there is a map $\ThS(\zeta
\oplus \eta) \to \ThS(\zeta)\Smash \ThS(\eta)$.  In particular, when
$\eta = 0$ in $KO(X)$ we get a map
\[
\Delta\co \ThS(\zeta) \to \ThS(\zeta) \Smash X_+
\]
(which makes $\ThS(\zeta)$ into a comodule spectrum over the coalgebra
$\Sigma^\infty X_+$ with its diagonal map).

Let $\gamma_{n,m}$ denote the tautological $n$-plane bundle over the
Grassmannian $Gr^+_n(\R^{n+m})$ of real oriented $n$-planes in
$\R^{n+m}$ (when $m=\infty$ we will write $\gamma_n \to BSO(n)$). One
defines the Madsen--Tillmann spectrum
\[
\MTSO (n) \defeq \ThS(-\gamma_n).
\]
Explicitly, the $(n+k)^{th}$ space of this spectrum is
$\Th(\gamma_{n,k}^\perp)$, and the spectrum structure maps are induced
from the identification
$\gamma_{n,k+1}|_{Gr^+_n(\R^{n+k})} \cong \gamma_{n,k}\oplus \R$. 
By the functoriality of Thom spectra, the classifying map $B \to BSO(d)$
for a rank $d$ vector bundle $\xi$ determines a map of spectra
\[
C_\xi\co \ThS(-\xi) \to \MTSO(d).
\]

Since $\gamma_{k+1}|_{Gr^+_k(\R^\infty)}\cong \gamma_{k} \oplus \R,$
the $KO$-theory classes $\{\gamma_n - \R^n\}_{n\in \N}$ assemble to a
class $\gamma$ on $BSO = \colim_n BSO(n)$ of virtual dimension zero.
One defines Thom's famous spectrum $\MSO$ as
\[
\MSO \defeq \ThS(\gamma)
\]
Concretely, it is the spectrum with $k^{th}$ space $\Th(\gamma_k)$ and
structure maps
\[
\Sigma \Th(\gamma_k) \cong
\Th(\gamma_{k+1}|_{Gr^+_k(\R^\infty)}) \hookrightarrow
\Th(\gamma_{k+1}).
\]

By the functoriality of Thom spectra, the classifying map $X \to BSO$
for a class $\xi \in KO(X)$ of virtual dimension $n$ produces a map of
Thom spectra
\[
C_{\xi}\co \ThS(\xi) \to \Sigma^{n}\MSO
\]
We call attention to two special cases of this map.  First, the
virtual bundle $-\gamma_n$ over $Gr^+_n(\R^\infty)$ gives a map of
spectra, $C_{-\gamma_n}\co \MTSO(n) \to \Sigma^{-n}\MSO.$ Second, the
external product $\gamma\times \gamma$ over $BSO\times BSO$ gives a
product map $\mu = C_{\gamma \times \gamma}\co \MSO \Smash \MSO \to
\MSO$ making $\MSO$ into a ring spectrum.

\subsection{Pre-transfers}
Given a bundle $\pi\co E\to B$, there is a \emph{pre-transfer} map
\[
PT_{\pi}\co \Sigma^\infty B_+ \to \ThS(-T^vE).
\]
It is constructed by choosing a fibrewise embedding of $E$ into a
trivial vector bundle over $B$ with a tubular neighbourhood,
collapsing everything outside the tubular neighbourhood to the
basepoint, and then identifying the tubular neighbourhood with the
open disc bundle of the fibrewise normal bundle.  The pre-transfer is
well defined up to homotopy.  

For later use we record the following factorisation property of
pre-transfers.

\begin{lemma}\label{PT-factorisation}
Given a smooth bundle of closed oriented $d$-manifolds $\pi\co E\to B$, the
diagram
\[
\begin{diagram}
\node{\mathbf{S}^0} \arrow{e,t}{\Delta \circ PT_{(B \to pt)}}
\arrow{s,l}{PT_{(E\to pt)}} \node{\ThS(-TB) \Smash B_+} \arrow{s,r}{id
\Smash PT_\pi} \\
\node{\ThS(-TE)} \arrow{e} \node{\ThS(-TB) \Smash \ThS(-T^v E)}
\end{diagram}
\]
(with bottom arrow determined by the splitting $TE \cong \pi^*TB
\oplus T^vE)$) commutes up to homotopy.
\end{lemma}
\begin{proof}
  Let $V, W$ be real vector spaces of sufficiently large dimension so
  that we can choose an embedding of $\phi\co B \hookrightarrow V$ and a
  fibrewise embedding $\psi\co E \hookrightarrow B\times W$ over $B$.
  Let $\Psi = (\phi \times id) \circ \psi\co E \hookrightarrow V\times
  W$.  Let $N\Psi, N\psi, N\phi$ be the respective normal bundles of
  $\Psi, \psi$ and $\phi$.  There is a canonical splitting $N\Psi
  \cong \pi^* N\phi \oplus N\psi$.  By choosing appropriate tubular
  neighbourhoods and identifications with the disc bundles one obtains
  a diagram
\[
\begin{diagram}
\node{S^V \Smash S^W} \arrow{s,l}{PT_{(E\to pt)}} 
\arrow{e,t}{PT_{(B\to pt)} \Smash id} \node{\Th(N\phi) \Smash S^W}
\arrow{e,t}{\Delta \Smash id}
\node{\Th(N\phi) \Smash B_+ \Smash S^W}  
\arrow{s,r}{id \Smash PT_\pi}\\
\node{\Th(N\Psi)} \arrow[2]{e} \node{\quad} \node{\Th(N\phi) \Smash \Th(N\psi)}
\end{diagram}
\]
that commutes on the nose.  (The PT maps here are unstable
representatives of the respective pre-transfers constructed from the
chosen embeddings and tubular neighbourhoods.)  The diagram in the
statement of the lemma is the stabilisation of this.
\end{proof}

\subsection{Oriented Bordism}
The classical Pontrjagin--Thom construction identifies the generalised
homology theory $MSO_*(X) = \pi_*(\MSO \Smash X_+)$ associated with the
spectrum $\MSO$ as \emph{oriented bordism} --- $MSO_n(X)$ is the
abelian group of oriented bordism classes of $n$-manifolds mapping to
$X$, with the group structure given by disjoint union.  Let us recall
how the half of the Pontrjagin--Thom construction that goes from
geometric cycles to homotopy classes can be stated efficiently in
terms of pre-transfers.

The oriented bordism class of a manifold $M^n$, as an element of
$MSO_n(pt) = \pi_n \MSO$, is represented by the
composition
\begin{equation}\label{eq:bordism-class1}
\mathbf{S}^0 = \Sigma^\infty pt_+ \stackrel{PT_{(M\to pt)}}{\longrightarrow}
\ThS(-TM) \stackrel{C_{-TM}}{\longrightarrow} \Sigma^{-n}\MSO.
\end{equation}
More generally, if $M$ is equipped with a map $f$ to a target space
$X$, then the class $[f\co M\to X] \in MSO_n(X)$ is represented by the
composition
\begin{equation}\label{eq:bordism-class2}
\mathbf{S}^0 \to \Sigma^\infty pt_+ \stackrel{PT_{(M\to pt)}}{\longrightarrow}
\ThS(-TM) \stackrel{\Delta}{\longrightarrow}
\ThS(-TM) \Smash M_+
 \stackrel{C_{-TM} \Smash f}{\longrightarrow} \Sigma^{-n}\MSO \Smash X_+.
\end{equation}

\subsection{Classifying maps for fibre bundles and the map $\alpha$}

Via (parametrised) Pontrjagin--Thom theory, the infinite loop space
$\Omega^\infty \MTSO(d)$ classifies bundles of smooth oriented closed
$d$-manifolds up to an an appropriate equivalence.  We shall make use
of these classifying maps to define the map $\alpha\co Bun_d \to
MSO_*(\Omega^\infty \MTSO(d))$ discussed in the introduction. The key
point is that this map sends a bundle to a class that determines both
the oriented bordism class of the total space of the bundle and all
MMM characteristic numbers of the bundle.

A closed oriented $d$-manifold $M$ determines an element $[M] \in
\pi_0 \MTSO (d) = \pi_0 \Omega^\infty \MTSO(d)$ by the obvious
modification of the construction from the preceding section.  More
generally, a bundle of $d$-manifolds $\pi\co E^{d+n} \to B^n$ determines
a homotopy class 
\[
\alpha_\pi\co B \to \Omega^\infty \MTSO(d).
\]
given by the adjoint of $C_{-T^vE} \circ PT_{\pi}$.  From this
description one sees immediately that $\alpha_\pi$ is natural (up to
homotopy) with respect to pullbacks.

Let us comment on connected components. Since $\Omega^\infty\MTSO(d)$
is a group-like $H$-space, each component can be canonically
identified (up to homotopy) with the identity component
$\Omega^\infty_0\MTSO(d)$.  By \cite[Theorem A.0.3]{Ebert-vanishing},
the group $\pi_0 \Omega^\infty\MTSO(d)$ fits into a split short exact
sequence
\[
0 \to \Z / k_{d+1}\Z \to \pi_0 \Omega^\infty\MTSO(d) \to \pi_d \MSO
\]
where $k_d$ is 0 if $d$ is odd, 1 if $d \equiv 0$ (mod 4), and 2 if
$d\equiv 2$ (mod 4). On each component of $B$ the map $\alpha_\pi$
lands in the component of $\Omega^\infty\MTSO(d)$ corresponding to the
class of the fibre over that component.

If the base $B$ is a closed oriented $n$-manifold, we can regard the
above map $\alpha_\pi$ as representing a class
\[
[\alpha_\pi] \in MSO_n (\Omega^\infty \MTSO(d))
\]
in the group of oriented bordism classes of $n$-manifolds mapping to
$\Omega^\infty \MTSO(d)$. The map $\alpha\co Bun_d \to MSO_*
(\Omega^\infty \MTSO(d))$ is defined by the rule
$(E \stackrel{\pi}{\to} B) \mapsto [\alpha_\pi]$.

\subsection{The oriented bordism class of the total space}
The oriented cobordism class of the total space $E$ is an element $[E]
\in MSO_{n+d}(pt)$.  Here we show how $[E]$ is determined by
$[\alpha_\pi]$ and hence give a canonical factorisation of $Tot\co
Bun_d \to MSO_{*+d}(pt)$ as $Bun_d \stackrel{\alpha}{\to}
MSO_*(\Omega^\infty \MTSO(d)) \stackrel{\widetilde{Tot}}{\to}
MSO_{*+d}(pt)$.

The counit of the $\Omega^\infty - \Sigma^\infty$ adjunction gives a
map of spectra $\sigma\co \Sigma^\infty \Omega^\infty \MTSO(d)_+ \to
\MTSO(d)$.  
Using the map $\MTSO(d) \to \Sigma^{-d} \MSO$ and the ring spectrum
product of $\MSO$, the composition
\[
\MSO\wedge \Omega^\infty \MTSO(d)_+ \to \MSO\wedge \MTSO(d) \to \MSO \wedge
\Sigma^{-d} \MSO \to \Sigma^{-d} \MSO
\]
induces an additive map
\[
\widetilde{Tot}\co MSO_n(\Omega^\infty \MTSO(d)) \to MSO_{n+d}(pt).
\]
\begin{lemma}\label{lemma-alphapi-determines-E}
The equality $\widetilde{Tot}([\alpha_\pi]) = [E]$ holds.
\end{lemma}
\begin{proof}
  The oriented bordism class $[E] \in \pi_{n+d} \MSO$ is represented
  by the map $C_{-TE} \circ PT_{(E \to pt)}\co \mathbf{S}^0 \to
  \Sigma^{-n-d}\MSO$.  Since $TE \cong T^vE \oplus \pi^*TB$, this map
  factors as
  \begin{equation}\label{proof-eq1}
  \mathbf{S}^0 \stackrel{PT_{(E \to pt)}}{\longrightarrow} \ThS(-\pi^* TB - T^v
  E) \longrightarrow \ThS(-TB) \Smash \ThS(-T^v E) \stackrel{\mu \circ
    C_{(-TB) \times (-T^vE)}}{\longrightarrow}
  \Sigma^{-n-d} \MSO.
  \end{equation}
  Since $\alpha_\pi$ is the adjoint of $C_{-T^vE} \circ PT_\pi$ and
  $\sigma$ is the adjoint of the identity map on $\Omega^\infty
  \MTSO(d)$, it follows that the diagram
\begin{equation}\label{proof-eq2}
\begin{diagram}
\node{\Sigma^\infty B_+} \arrow{s,l}{PT_\pi} \arrow{e,t}{\Sigma^\infty
  \alpha_\pi} \node{\Sigma^\infty \Omega^\infty \MTSO(d)_+} \arrow{s,r}{\sigma} \\
\node{\ThS(-T^vE)} \arrow{e,t}{C_{-T^vE}} \node{\MTSO(d)}
\end{diagram}
\end{equation}
commutes up to homotopy.  The claim now follows directly from the
definitions, the factorisation \eqref{proof-eq1}, the diagram
\eqref{proof-eq2} and Lemma \ref{PT-factorisation}.
\end{proof}

\section{Generalised MMM classes via Madsen--Tillmann spectra}

\subsection{Generalised MMM classes}
Here we recall how one can think of the generalised MMM classes of
$d$-dimensional fibre bundles universally as classes in the cohomology
of $\Omega^\infty \MTSO(d)$.

\begin{proposition}\label{univ-classes}
  There is a canonical injective morphism $\alg_d \hookrightarrow
  H^*(\Omega^\infty_0 \MTSO(d); \Z)$ of bialgebras, and after
  tensoring with $\Q$ it becomes an isomorphism.  
\end{proposition}

\begin{proof}
  By the Thom isomorphism, the (integral) spectrum cohomology of
  $\MTSO(d)$ is a free $H^*(BSO(d);\Z)$-module on one generator $u$
  (the Thom class) in degree $-d$.  The map $\sigma$ from the previous
  section induces a map on cohomology,
  \[
  \sigma^*\co H^* (\MTSO(d);\Z) \to H^*(\Omega^\infty \MTSO(d);\Z).
  \]
  Let $\sigma^*_0$ denote the composition of $\sigma^*$ followed by
  restriction to the identity component $\Omega^\infty_0 \MTSO(d)$;
  $\sigma^*_0$ is injective in positive degrees and it is trivially
  zero in negative degrees.  The morphism of the proposition is
  defined by sending the MMM class $\widehat{X}=\pi_!X$ to
  $\sigma^*_0 (uX)$.  With rational coefficients the positive degree
  part of the image of $\sigma^*_0$ coincides with the space of
  primitives and this subspace freely generates the cohomology
  $H^*(\Omega^\infty_0 \MTSO(d);\Q)$ as a graded commutative algebra
  by the Milnor--Moore theorem \cite[p. 263]{Milnor-Moore}.
\end{proof}

Note that in positive degrees the image of $\sigma^*$ is contained in
the subspace of $\pi_0 \Omega^\infty \MTSO(d)$-invariants.  The
cohomology of the component $\Omega^\infty_0 \MTSO(d)$ is canonically
isomorphic to the $\pi_0$-invariant subspace of the cohomology of
$\Omega^\infty \MTSO(d)$.  Hence we shall identify $\alg_d\otimes \Q$
with the $\pi_0$-invariants in $H^*(\Omega^\infty \MTSO(d);\Q)$.

Given a bundle $\pi\co E \to B$ of closed oriented manifolds and a
vector bundle characteristic class $X \in H^*(BSO(d);\Z)$, we have a
cohomology class $uX \in H^*(BSO(d);\Z)[u] \cong H^*(\MTSO(d);\Z)$ and
also a primitive generalised MMM class $\pi_! X(T^vE)$.  This
proposition justifies our identifications above.

\begin{proposition}\label{prop:char-classes-pullback}
For any bundle $\pi\co E \to B$ there is an equality, 
$(\alpha_\pi)^*\sigma^*(uX) = \pi_! X(T^vE).$  
\end{proposition}
\begin{proof}
  This follows immediately from the commutativity of the diagram
  \eqref{proof-eq2} and the fact that the pre-transfer in cohomology,
  $PT^*_\pi$, composed with the Thom isomorphism for $-T^v E$ is equal
  to the pushforward map $\pi_!$.
\end{proof}

\subsection{MMM characteristic numbers from $[\alpha_\pi]$}
We now show how the generalised MMM characteristic numbers for a
bundle $\pi\co E^{n+d}\to B^n$ are determined by the class $[\alpha_\pi]
\in MSO_n(\Omega^\infty \MTSO(d))$ constructed earlier.  Hence a
generalised MMM characteristic number $x^\sharp$ admits a canonical
factorisation $Bun_d \stackrel{\alpha}{\to} MSO_*(\Omega^\infty
\MTSO(d)) \to \Q$.

Let $\mathbf{H}\Z$ denote the Eilenberg--MacLane spectrum representing ordinary
cohomology.  The Thom class of $MSO$ determines a map of ring spectra
\[
th\co \MSO \to \mathbf{H}\Z
\]
inducing a natural transformation of homology functors $th_*\co MSO_*(-)
\to H_*(-; \Z)$, which can be described geometrically as follows. Any
element of $MSO_k(X)$ is represented by a closed oriented $k$-manifold
$M$ and a map $f\co M\to X$. 

\begin{lemma}\label{thom-class-lemma}
The homomorphism $th_*$ sends the class $[(M^k \to X)]$ to
the homology class given by the image of the fundamental class of $M$; i.e.,
$th_*([(M\stackrel{f}{\to} X)]) = f_*[M] \in H_k(X; \Z)$.
\end{lemma}
\begin{proof}
  Let $\mu_M$ denote the fundamental cohomology class of $M$.  Let
  $u$ denote the Thom class of $-TM$.  It suffices to show that
  $\langle th_*[M\stackrel{id}{\to} M], \mu_M \rangle = 1$.  Moreover,
  by the naturality of Thom classes and the definition of the
  Kronecker pairing, it suffices to show that the composition
\[
\Sigma^\infty pt_+ \stackrel{PT_{(M\to pt)}}{\longrightarrow}
\ThS(-TM) \stackrel{\Delta}{\longrightarrow} \ThS(-TM) \Smash M_+
\stackrel{u\Smash \mu_M}{\longrightarrow} \mathbf{H}\Z \Smash \mathbf{H}\Z \stackrel{prod}{\longrightarrow} \mathbf{H}\Z
\]
represents $1 \in H^0(pt;\Z)$.  This holds because $\Delta^*(u \Smash
\mu_M) = u \mu_M$ and $PT_{(M \to pt)}^* (u\cdot -)$ is integration
over $M$.
\end{proof}

\begin{lemma}\label{alphapi-determines-nums}
  Let $c \in \mathrm{Sym} \left( H^{> 0}(\MTSO(d);\Z) \right) \subset H^*(\Omega^\infty \MTSO(d);
  \mathbb{Z})$ be a polynomial in the universal generalised MMM
  classes.  Then, $c^\sharp(E^{n+d}\stackrel{\pi}{\to} B^n) = \langle
  c, th_*[\alpha_\pi] \rangle.$
\end{lemma}
\begin{proof}
We have
\begin{align*}
c^\sharp(E\to B) & = \langle (\alpha_\pi)^* c, [B] \rangle \\
                & = \langle c, (\alpha_\pi)_*[B] \rangle \\
                & = \langle c, th_*[\alpha_\pi] \rangle,
\end{align*}
where the first equality is immediate from Proposition
\ref{prop:char-classes-pullback}, and the last equality comes from
Lemma \ref{thom-class-lemma}.
\end{proof}

\section{The image of $\alpha$ and the proof of Theorem
  \ref{main-thm}}\label{section:alpha}

We now show how Theorem \ref{submain-thm} together with Ebert's
results on the algebraic independence of the generalised MMM classes
implies Theorem \ref{main-thm} and the $Bun^\C_d$ part of Theorem
\ref{complex-thm}.

\subsection{The real oriented case}
In studying the image of $\alpha$ it will be convenient to dualize and
instead study the kernel of the linear map
\[
\alpha^\vee \co
\Hom(Bun_d, \Q) \to \Hom(MSO_*(\Omega^\infty \MTSO(d)),\Q).
\]

According to Proposition \ref{univ-classes} the space of generalised
MMM classes $\alg_d\otimes\Q$ can be canonically identified with a
subspace of $\Hom(MSO_*(\Omega^\infty \MTSO(d)),\Q)$.  By Theorem
\ref{submain-thm} an element $x$ in the subspace $\alg_d\otimes \Q$
lies in the image of
\[
(\widetilde{Tot})^\vee\co \Hom(MSO_{*+d}(pt), \Q) \to \Hom(MSO_*(\Omega^\infty
\MTSO(d)),\Q).
\]
if and only if $x$ is of the form $\widehat{X}$ for $X$ in the space
$NP_d$ of order $d$ near-primitives (restricted from $BSO$ to
$BSO(d)$).

\begin{lemma}
  The restriction of $\alpha^\vee$ to the subspace $\alg_d\otimes\Q$ is
  injective when $d$ is even; when $d$ is odd there is a kernel $K$
  and it is the ideal in $\alg_d\otimes\Q$ generated by the
  generalised MMM classes of the form $\widehat{X}$ for $X$ a
  homogeneous component of the Hirzebruch $L$-class.
\end{lemma}
\begin{proof}
First suppose $d$ is even.  Ebert's algebraic independence result
\cite[Theorem B]{Ebert-alg-indep} asserts that for any nonzero class
$x\in \alg_d\otimes \Q$ (which we take to be of degree $m$) there
exists a $d$-manifold $Z$ such that the evaluation of $x$ on the
tautological $Z$-bundle over $B\Diff(Z)$ is nontrivial in rational
cohomology.  Since the map $MSO_*(B\Diff(Z))\otimes\Q \to H_*(Z;\Q)$
is surjective, there exist a closed oriented $m$-manifold $M$ and a
map $f\co M \to B\Diff(Z)$ such that $\langle f^*x, [M] \rangle$ is
nonzero.  Hence $x^\sharp$ evaluates nontrivially on the induced
$Z$-bundle over $M$ and so $\alpha^*(x^\sharp) \neq 0$.

When $d$ is odd, Ebert's theorem asserts the existence of a manifold $Z$
for which $x$ is nonzero on $B\Diff(Z)$ as above if and only if if $x$
is nonzero in the quotient $\alg_d/K$, where $K$ is the ideal
generated by the classes of the form $\widehat{X}$ for $X$ a
homogeneous component of the Hirzebruch $L$-class.
\end{proof}

Theorem \ref{main-thm} now follows immediately from the above lemma
together with Theorem \ref{submain-thm}, which we shall prove below in
Section \ref{section:calculation}.

\subsection{The complex case}

The argument in this case is identical to the argument in the real
case.  The injectivity of $(\alpha^\C)^\vee |_{\alg_d^\C\otimes \Q}$ follows from
\cite[Theorem C]{Ebert-alg-indep}.  In this setting there is no
distinction between even and odd $d$.

\subsection{The Madsen--Weiss Theorem and $d=2$}
The discussion here is not necessary for any of the results in this
paper, but we include it for completeness.  When $d=2$ the
Madsen--Weiss theorem \cite{Madsen-Weiss} implies that the image of
$\alpha$ is quite large.  Let $\pi_0$ denote the group of path
components of $\Omega^\infty\MTSO(2)$.  (This group is isomorphic to
$\Z$ and corresponds to the genus of the fibres of a surface bundle.)

\begin{proposition}
The map $Bun_2 \stackrel{\alpha}{\to} \MSO_*(\Omega \MTSO(2)) \to
\MSO_*(\Omega^\infty \MTSO(2))_{\pi_0}$ is surjective.
\end{proposition}

\begin{proof}
Let $\Sigma_{g,1}$ be a compact connected oriented surface of genus
$g$ with one boundary component, and let
$\Diff_\partial(\Sigma_{g,1})$ denote the topological group of
orientation-preserving diffeomorphisms that restrict to the identity
on the boundary.  Gluing the boundary of the surface to one boundary a
torus with two boundary components gives a stabilisation map
$\Diff_\partial (\Sigma_{g,1}) \to \Diff_\partial (\Sigma_{g+1,1})$.

For each $g$ there is a universal bundle $\pi(g) \co E_g \to
B\Diff_\partial(\Sigma_{g,1})$ and thus an associated map
\[
\alpha_{\pi(g)} \co B\Diff_\partial(\Sigma_{g,1}) \to \Omega^\infty\MTSO(2)
\]
One can adjust these maps to all land in the identity component
$\Omega^\infty_0\MTSO(2)$ by translating and then check that they are
compatible with the above stabilisation up to homotopy.  By the
Madsen--Weiss Theorem \cite{Madsen-Weiss} (see also \cite{GMTW} and
\cite{Randal-W-Galatius}), the resulting map
\[
\alpha_{\infty}\co \hocolim_{g\to \infty} B\Diff_\partial
(\Sigma_{g,1}) \to \Omega^\infty_0\MTSO(2)
\]
is an isomorphism in any generalised cohomology theory, including
$MSO_*(-)$.  

Since $B\Diff_\partial (\Sigma_{g,1})$ classifies bundles with fibre
$\Sigma_{g,1}$ (and trivial boundary bundle), the map $\alpha\co Bun_2
\to MSO_*(\Omega^\infty\MTSO(2))$, when restricted to the subspace
$Bun_2' \subset Bun_2$ of bundles that admit a framed section, factors
as
\[
Bun_2' \to MSO_*(\coprod_g  B\Diff(\Sigma_{g,1})) \to MSO_*(\Omega^\infty\MTSO(2)).
\]
It now follows from the Madsen--Weiss Theorem that $\alpha$ is surjective
onto the $\pi_0$-coinvariants.
\end{proof}

When $d$ is even and greater than 4, Galatius and Randal-Williams have
announced in \cite{RG2} a result along the lines of the Madsen--Weiss
Theorem.  This result implies a large lower bound on the size of the
image of $\alpha$.

\section{Reformulation in terms of stable
  homotopy}\label{section:calculation}

Here we give the proofs of Theorem \ref{submain-thm} and Theorem
\ref{complex-thm} part (1) by reformulating the factorisation
question at the level of spectra and carrying out a rational
calculation.

\subsection{The real case}
The problem is to decide which generalised MMM classes $x\in
\alg_d\otimes\Q$ have the property that the associated characteristic
number $x^\sharp$ admits a factorisation
\begin{equation}\label{mini-fact-diagram}
\begin{diagram}
  \node{MSO_*(\Omega^\infty \MTSO(d))\otimes\Q} \arrow{se,b}{x^\sharp}
  \arrow{e,t}{\widetilde{Tot}}
  \node{MSO_{*+d}(pt)\otimes\Q} \arrow{s,r,..}{y}\\
  \node[2]{\Q.}
\end{diagram}
\end{equation}
We now reformulate this diagram at the level of spectra.  The set of
linear maps $MSO_*(pt)\otimes\Q \to \Q$ is in canonical bijection with
$H^*(\MSO;\Q)$ and thus with homotopy classes of maps $\MSO \to
\mathbf{H}\Q$.  Given a bundle $\pi\co E^{n+d} \to B^n$, the class
$[\alpha_\pi] \in MSO_n(\Omega^\infty \MTSO(d))$ is represented by a
degree $n$ map of spectra
\[
\mathbf{S}^0 \to \MSO\wedge \Omega^\infty \MTSO(d)_+.
\]
As described earlier, the oriented bordism class of the total space of
the bundle is represented by the homotopy class of the composition
\[
\mathbf{S}^0 \stackrel{\alpha_\pi}{\to} \MSO\Smash \Omega^\infty \MTSO(d)_+ \stackrel{id\Smash \sigma}{\to} \MSO\Smash \MTSO(d) \to
\MSO\Smash \MSO \stackrel{prod}{\to} \MSO,
\]
which is a map of degree $n+d$.  It is elementary to see that, given a
class $x \in H^n(\Omega^\infty \MTSO(d);\Q)$, the characteristic
number $x^\sharp(E^{n+d} \stackrel{\pi}{\to} B^n)$ is represented by
the homotopy class of the composition
\[
\mathbf{S}^0 \stackrel{[\alpha_\pi]}{\to} \MSO\wedge \Omega^\infty \MTSO(d)_+
\stackrel{th}{\to} \mathbf{H}\Q\wedge \Omega^\infty \MTSO(d)_+ \stackrel{1\wedge
  x}{\to} \mathbf{H}\Q\wedge \mathbf{H}\Q \stackrel{prod}{\to} \mathbf{H}\Q
\]
which is of degree zero and hence gives an element of $\Q$.  We
therefore have the following.

\begin{lemma}\label{lemma:spectra-diagram}
The problem of factorisation in diagram \eqref{mini-fact-diagram} is
thus equivalent to the problem of deciding for which $x$ there exists
some $y\co \MSO \to \mathbf{H}\Q$ such that the diagram
\[
\begin{diagram}
\node{\MSO\wedge \Omega^\infty \MTSO(d)_+} \arrow{e} \arrow{s} \node{\MSO
\wedge \MTSO(d)} \arrow{e} \node{\MSO\wedge \MSO} \arrow{e} \node{\MSO}
\arrow{s,r,..}{y}\\
\node{\mathbf{H}\Q\wedge \Omega^\infty \MTSO(d)_+} \arrow[2]{e,t}{1\wedge x} \node[2]{\mathbf{H}\Q \wedge
\mathbf{H}\Q} \arrow{e} \node{\mathbf{H}\Q}
\end{diagram}
\]
commutes up to homotopy.
\end{lemma}

\begin{proof}[Proof of Theorem \ref{submain-thm}]
  The counter-clockwise composition in the diagram of Lemma
  \ref{lemma:spectra-diagram} represents a class in
\[
H^*(\MSO\wedge\Omega^\infty \MTSO(d);\Q) \cong H^*(\MSO;\Q)\otimes H^*(\Omega^\infty
\MTSO(d);\Q)
\]
that one immediately identifies as $th \otimes x$.  Thus existence of
such a $y$ is equivalent to existence of a class $y\in H^*(\MSO;\Q)$
that pulls back along the top row to $th\otimes x$.  By Lemma
\ref{restriction-lemma} (proved in the sequel), this is equivalent to
$y$ being the image under the Thom isomorphism of a class $y' \in
H^*(BSO;\Q)$ that is near-primitive of order $d$ and whose restriction
to $BSO(d)$ maps to $x$ under the Thom isomorphism.  This concludes
the proof of Theorem \ref{submain-thm}.  
\end{proof}

\subsection{The complex case}

Suppose now that $\pi\co E^{n+d} \to B^{n}$ is a holomorphic fibre
bundle in the category of closed almost complex manifolds (the
superscripts refer to the complex dimensions).  Then the fibrewise
tangent bundle has structure group $U(d)$ and we have a classifying
map
\[
\alpha_\pi\co B \to \Omega^\infty \MTU(d).
\]
The bordism class of this map represents an element $[\alpha_\pi] \in
MU_{2n}(\Omega^\infty \MTU(d))$, and this defines an additive map
$\alpha^\C\co Bun_d^\C \to MU_*(\Omega^\infty \MTU(d))$. As in the
oriented case, there is a map
\[
\widetilde{Tot}\co MU_*(\Omega^\infty \MTU(d)) \to MU_{*+2d}(pt)
\]
The total space $E$ represents an element of $MU_{n+2d}(pt)$ and
$\widetilde{Tot}\co [\alpha_\pi] \mapsto [E]$.  

We now consider the diagram:
\[
\begin{diagram}
\node{\MU\wedge \Omega^\infty \MTU(d)_+} \arrow{e} \arrow{s} \node{\MU
\wedge \MTU(d)} \arrow{e} \node{\MU\wedge \MU} \arrow{e} \node{\MU}
\arrow{s,r,..}{y}\\
\node{\mathbf{H}\Q\wedge \Omega^\infty \MTU(d)_+} \arrow[2]{e,t}{1\wedge x} \node[2]{\mathbf{H}\Q \wedge
\mathbf{H}\Q} \arrow{e} \node{\mathbf{H}\Q.}
\end{diagram}
\]

\begin{proof}[Proof of Theorem \ref{complex-thm}, part (1)]
  As in the real case, existence of a $y$ making this diagram commute
  up to homotopy is equivalent to existence of a class $y\in
  H^*(\MU;\Q)$ that pulls back along the top row to $th\otimes x$, and
  by Lemma \ref{restriction-lemma} (proved in the sequel), this is
  equivalent to $y$ being the image under the Thom isomorphism of a
  class $y' \in H^*(BU;\Q)$ that is near-primitive of order $2d$ and
  whose restriction to $BU(d)$ maps to $x$ under the Thom
  isomorphism. This concludes the proof of the part of Theorem
  \ref{complex-thm} referring to maps out of $MU_*(\Omega^\infty
  \MTU(d))$.  
\end{proof}

\section{Near-primitives}\label{classify-near-primitives}

In this section we shall prove Theorem \ref{near-prim-thm} which gives
a complete description of the near-primitives in $H^*(BSO;\Q)$ and
$H^*(BU;\Q)$.  We also prove Lemma \ref{restriction-lemma} showing
that near-primitives of order $d$ on $BSO$ or order $2d$ on $BU$ can
be effectively detected by restricting to $BSO(d)$ or $BU(d)$
respectively.

\subsection{Recollections}
Recall that, as an algebra,  $H^*(BU;\Q)$ is a polynomial algebra with
a single generator in each even degree.  These generators can be taken
to be the Chern classes $c_i \in H^{2i}(BU;\Q)$.  In terms of these
generators, the coproduct is determined by the formula
\[
\Delta(c_i) = \sum _{i=0}^n c_i \otimes c_{n-i}
\]
(where $c_0=1$). The cohomology of $BU(d)$ is isomorphic to $\Q[c_1,
\ldots, c_d]$, and the restriction from $BU$
to $BU(d)$ kills the ideal generated by $\{ c_{d+1}, c_{d+2}, \ldots\}$.

In the real case, $H^*(BSO;\Q)$ is a polynomial algebra with
generators the Pontrjagin classes $p_i \in H^{4i}(BSO;\Q)$.  The
coproduct is determined by the formula
\[
\Delta(p_i) = \sum _{i=0}^n p_i \otimes p_{n-i}
\]
(where $p_0=1$).  The cohomology of $BSO(d)$ is isomorphic to $\Q[p_1,
\ldots, p_{\lfloor d/2\rfloor}]$ if $d$ is odd and $\Q[p_1, \ldots,
p_{d/2}, e]/(e^2 = p_{d/2})$ if $d$ is even.  The restriction from $BSO$
to $BSO(d)$ kills the ideal generated by $\{ p_{\lfloor d/2 \rfloor +
  i}\}_{i\geq 1}$.

Let $H$ be either $H^*(BU;\Q)$ or $H^*(BSO;\Q)$.  To describe
near-primitives in $H$ it is convenient to work with a set of
primitive generators.  By the structure theorem for Hopf algebras, the
dual of $H$ is polynomial, from which it follows that the space of
primitives in degree $i$ has dimension 1 if $H^i$ is non-zero and has
dimension zero otherwise.  Thus there is a primitive $Q_i$ (unique up
to scalars) in each degree $2i$ in the complex case, and $4i$ in the
real case. These primitives can be taken to be the components of the
Chern character and Pontrjagin character respectively.  The primitives
freely generate $H$ as a graded commutative algebra, so that $H$ is
the polynomial algebra in $Q_1, Q_2, \dots$.

Since the coproduct $\Delta$ is connected, if $x$ is an element of
positive degree then $\Delta(x)$ is equal to $x\otimes 1 + 1 \otimes
x$ plus terms of bidegree $(p,q)$ with $p,q > 0$. For convenience, we
shall work with the \emph{reduced coproduct} $\overline{\Delta}$
defined by $\overline{\Delta} = \Delta - 1\otimes id - id \otimes 1$
in positive degrees, and $\overline{\Delta} = 0$ on elements of degree
zero.  Observe that if $x\in H$ is an element of degree $\geq d$
then $x$ is a near-primitive of order $d$ if and only it is in the
kernel of the map
\begin{equation}\label{near-prim-defining-map}
H \stackrel{\overline{\Delta}}{\to} H\otimes H \stackrel{id \otimes
  \mathrm{proj}}{\longrightarrow} H\otimes H^{\geq d},
\end{equation}
where $\mathrm{proj}$ projects the summands in degree $< d$ to zero.

\subsection{Classification of near-primitives in $H^*(BSO;\Q)$ and $H^*(BU;\Q)$}

We now prove Theorem \ref{near-prim-thm}, restated below.  We continue
to let $H$ denote either $H^*(BSO;\Q)$ or $H^*(BU;\Q)$.
\begin{theorem}\label{kerbasis}  
  The space of near-primitives of order $d$ in $H$ has a basis
  consisting of those monomials in the primitives $Q_i$ such that all
  proper factors are of degree strictly less than $d$.  Explicitly, the degree $m$
  component of the space is equal to the degree $m$ part of
  $\Q[Q_i \:\:|\:\: m-d < |Q_i| < d ]\oplus \Q Q_m$.
\end{theorem}

In a fixed degree $m$ (which we take to be even since $H$ is
concentrated in even degrees), let us consider how the space of
near-primitives depends on $d$.  If $d < m/2$ then the only
near-primitives are the ordinary primitives since the set of $i$ for
which $m-d < |Q_i| < d$ is empty.  When $d = m/2$, if there is a $Q_i$
of degree $d-1$ then $Q_i^2$ is a near-primitive that is not
primitive.  As $d$ increases more near-primitives appear, until $d=m$
at which point everything is in the kernel of the map
\eqref{near-prim-defining-map}.  In general, the dimension of the
space of near-primitives in degree $m$ can easily be recovered from
the Poincar\'{e} series for the polynomial algebra on $\{Q_i: m-d <
|Q_i| < d \}$.

Theorem \ref{kerbasis} follows directly from the following property of
the coproduct in $H$.
\begin{lemma}\label{primcoprod}
  Let $x \in \Q[Q_1, Q_2, \dots]$. The monomial summands of
  $\bar{\Delta}(x)$ correspond to proper factorizations of summands of
  $x$. In particular, $\bar{\Delta}(x)$ has a summand $y \otimes
  z$ with $|z| = k$ if and only if $x$ has a summand with a degree $k$ factor.
\end{lemma}
\begin{proof}
  Since each $Q_i$ is primitive, we can easily calculate the coproduct
  of a monomial explicitly:
\[
\Delta(Q_1^{f_1} Q_2^{f_2} \cdots) = \sum_{i_1, i_2, \dots} \left({f_1
    \choose i_1}{f_2 \choose i_2}\cdots\right) (Q_1^{i_1} Q_2^{i_2}
\cdots) \otimes (Q_1^{f_1-i_1} Q_2^{f_2-i_2} \cdots) ,
\]
where the summation runs over all sequences $i_1, i_2, \dots$ with $0
\leq i_1 \leq f_1$, $0 \leq i_2 \leq f_2, \dots$.  Hence if the
monomial $(Q_1^{f_1} Q_2^{f_2} \cdots)$ has a degree $k$ factor, say
$Q_1^{i_1} Q_2^{i_2} \cdots$, then $\Delta(Q_1^{f_1} Q_2^{f_2}
\cdots)$ will have a summand $(Q_1^{i_1} Q_2^{i_2} \cdots) \otimes
(Q_1^{f_1-i_1} Q_2^{f_2-i_2} \cdots)$ (with coefficient ${f_1 \choose
  i_1}{f_2 \choose i_2}\cdots$). Consequently, $\bar{\Delta}(Q_1^{f_1}
Q_2^{f_2} \cdots)$ will have a summand $y \otimes z$ with $|z|=k$.

Conversely, if $\bar{\Delta}(Q_1^{f_1} Q_2^{f_2} \cdots)$ has a
summand $y\otimes z$ with $|z|=k$, then $z$ must be $(Q_1^{i_1}
Q_2^{i_2} \cdots)$ for some sequence $i_1, i_2, \dots$ with $i_1 \leq
f_1$, $i_2 \leq f_2$, etc. Hence $(Q_1^{i_1} Q_2^{i_2} \cdots)$ is a
factor of $(Q_1^{f_1} Q_2^{f_2} \cdots)$.

From the above formula it is also clear that two different monomials
$Q_1^{f_1} Q_2^{f_2} \cdots$ and $Q_1^{g_1} Q_1^{g_2} \cdots$ cannot
have the same summand in their images under $\Delta$, since adding the
exponents reveals the source of the summand.
\end{proof}

\subsection{Restricting to $BU(d)$ or $BSO(d)$}

Let $i$ denote either of the inclusions $BSO(d) \hookrightarrow BSO$
or $BU(d) \hookrightarrow BU$.  Let $\mathrm{proj}$ denote either the
projection $H^*(BSO;\Q) \to H^{\geq d}(BSO;\Q)$ or the projection
$H^*(BU;\Q) \to H^{\geq 2d}(BU;\Q)$.  The following lemma asserts that
near-primitives of order $d$ on $BSO$ can be faithfully detected on
$BSO(d)$ in an appropriate sense, and likewise for near-primitives of
order $2d$ and $BU(d)$.
\begin{lemma}\label{restriction-lemma}
  There is an equality of kernels,
\[
\ker(1\otimes(i^*\circ\mathrm{proj}))\circ\bar{\Delta} =
\ker(1\otimes\mathrm{proj})\circ\bar{\Delta} 
\]
\end{lemma}
\begin{proof}
  We first prove the complex case and then describe the modifications
  needed for the real case.  Let $x \in
  \ker(1\otimes(i^*\circ\mathrm{proj}))\circ\bar{\Delta}$, i.e., $x
  \in \Q[Q_1, Q_2, \dots]$ is such that the only summands $y \otimes
  z$ in $\bar{\Delta}(x)$ in which $|z| \geq 2d$ have $z \in
  \mathrm{ker}(i^*) = \langle c_{d+1}, c_{d+2}, \dots \rangle$.  If
  $x$, as a polynomial in $Q_1, Q_2, \ldots$, has a monomial summand
  that features $Q_j$ as a factor with $|Q_j| \geq 2d$, then, by Lemma
  \ref{primcoprod}, that summand will give rise to a term in
  $\bar{\Delta}(x)$ of the form $w \otimes Q_j$. Since $|Q_j| \geq 2d$
  we would need $Q_j \in \langle c_{d+1}, c_{d+2}, \dots \rangle$ to
  avoid contradicting the hypotheses on $x$. However, writing the
  primitives as polynomials in the Chern classes, $Q_j$ contains
  $c_1^j$ as a summand with coefficient $1/(j!)$ by Newton's
  identities, so $Q_j \not\in \langle c_{d+1}, c_{d+2}, \dots
  \rangle$. Hence $x$ can have no summand admitting a proper factor
  $Q_j$ with $|Q_j| \geq 2d$.  Hence $\bar{\Delta}(x)$ features no
  summands $y \otimes z$ with $z \in \langle c_{d+1}, c_{d+2}, \dots
  \rangle$, and so $x \in
  \ker(1\otimes\mathrm{proj})\circ\bar{\Delta}$.  The opposite
  inclusion is trivial.

  The proof in the real case is nearly identical.  The only difference
  is that now $Q_j$ sits in degree $4j$ and, when written as a
  polynomial in the Pontrjagin classes, contains a term $p_j^1$ with
  coefficient $1/(j!)$.
\end{proof}

\bibliographystyle{amsalpha}
\bibliography{biblio}

\end{document}